\documentclass[leqno]{amsart}
\usepackage[T1,T2A]{fontenc}
\usepackage[utf8]{inputenc}
\usepackage[english]{babel}
\usepackage{amsthm}
\usepackage{color}
\usepackage{amssymb}
\usepackage{amsmath,amsfonts,enumerate}
\usepackage{amsthm}
\usepackage{amsmath}
\usepackage{graphicx}
\usepackage{hyperref}

\numberwithin{equation}{section}

\newtheorem{theorem}{Theorem}
\newtheorem{lemma}[theorem]{Lemma}

\newtheorem{definition}[theorem]{Definition}

\numberwithin{equation}{section}

\begin{document}

\title{Conditional expectation on non-commutative $H^{(r,s)}_{p}(\mathcal A;\ell_{\infty})$ and
$H_{p}(\mathcal A;\ell_{1})$ spaces : semifinite case}

\author{$^{1}$D. Dauitbek}
\email{dauitbek@math.kz}
\address{Al-Farabi Kazakh National University, 050040, Almaty, Kazakhstan;
Institute of Mathematics and Mathematical Modeling, 050010, Almaty, Kazakhstan.
}

\author{$^{2}$K. Tulenov}
 \email{tulenov@math.kz}
\address{Al-Farabi Kazakh National University, 050040, Almaty, Kazakhstan;
Institute of Mathematics and Mathematical Modeling, 050010, Almaty, Kazakhstan.}

\thanks{Corresponding authors: $^{1,2}$Al-Farabi Kazakh National University, 050040, Almaty, Kazakhstan;
$^{1,2}$Institute of Mathematics and Mathematical Modeling, 050010, Almaty, Kazakhstan. email: $^{1}$dauitbek@math.kz and $^{2}$tulenov@math.kz}


\subjclass[2010]{46L51, 46L52.}
\keywords{semifinite von Neumann algebra, semifinite subdiagonal algebra, non-commutative Hardy spaces, conditional expectation.}
\date{}
\begin{abstract} In this paper we investigate the conditional expectation on the non-commutative $H^{(r,s)}_{p}(\mathcal A;\ell_{\infty})$ and
$H_{p}(\mathcal A;\ell_{1})$ spaces associated with semifinite subdiagonal algebra, and prove  the contractibility of the underlying conditional
expectation on these spaces.
\end{abstract}

\maketitle

\section{Introduction}

Let $\mathcal{M}$ be a semifinite von Neumann algebra equipped with a faithful normal semifinite trace $\tau$ and let $\mathcal D$ be a subalgebra of the $\mathcal{M}.$ Let $L_{p}(\mathcal M)$ and $L_{p}(\mathcal D)$ be the corresponding non-commutative $L_p$-spaces, respectively. A conditional expectation $\mathbb{E}:\mathcal M\rightarrow\mathcal D$ is a unique
 normal faithful map such that
 $\tau\circ\mathbb{E}=\tau$.
It is well-known that the conditional expectation  $\mathbb{E}$ extends to a contractive projection from $L_{p}(\mathcal M)$ onto $L_{p}(\mathcal D)$ for every $1\leq p\leq\infty.$ In
general, $\mathbb{E}$ cannot be continuously extended to
$L_{p}(\mathcal M)$ for $p<1.$
In \cite{A}, Arveson introduced the notion of finite, maximal, subdiagonal algebras $\mathcal A$ of $\mathcal M,$ as non-commutative analogues of weak*-Dirichlet algebras, for von Neumann algebra $\mathcal M$ with a faithful normal finite trace. Subsequently several authors studied the (non-commutative) $H_p$-spaces associated with such algebras. For references see
\cite{A, BXu, B1, BL12008, BL2006, Ex, Ji, L2005, MW, P, PX2003, S} (see also \cite{BS, BTD, STZ, T1,T2} for recent studies) and many other sources, whereas more
references on previous works can be found in the survey paper
\cite{PX2003}. It was proved in \cite{BXu} that $\mathbb{E}$
is a contractive projection from $H_{p}(\mathcal A)$ onto $L_{p}(\mathcal D)$ for $p<1$.
Moreover, in \cite{B1} it was studied non-commutative $H_p$-spaces associated with semifinite subdiagonal aslgebra. The main objective of such paper is to obtain a generalization from above cited papers for the semi-finite case.

In \cite{BTD}, authors obtained that  there is a contractive projection from
$H_{p}^{(r,s)}(\mathcal A;\ell_{\infty})$ onto $L_{p}^{(r,s)}(\mathcal D;\ell_{\infty})$ for
$0<p,r,s\leq\infty$ (respectively, from $H_{p}(\mathcal A;\ell_{1})$ onto $L_{p}(\mathcal D;\ell_{1})$ for $0<p\leq\infty$) and when subdiagonal algebra is finite.
 The main goal of this paper is to extend those results to the semifinite case.
Addressing precisely this framework, our main results, Theorems \ref{elinfty} and \ref{el1}, provide a complete description of these results, thereby complementing \cite[Theorem 3 and Theorem 4]{BTD}. The main technical tool of the paper, which enables us to obtain such results, is provided by Proposition 3.1 in \cite{B1}.

\section{Preliminaries}\label{Preliminaries}

Let $\mathcal{M}$ be a semifinite von Neumann algebra on a separable Hilbert space $H$ equipped with a faithful normal semifinite trace $\tau.$
A closed and densely defined operator $x$ affiliated with $\mathcal{M}$ is called $\tau$-measurable if $\tau(E_{|x|}(s,\infty))<\infty$ for sufficiently large $s.$ We denote the set of all $\tau$-measurable operators by
$S(\mathcal{M},\tau).$ Let $\mathcal D$ be a von Neumann
 subalgebra of $\mathcal M$, and let $\mathbb{E}:\mathcal M\rightarrow\mathcal D$ be the unique
 normal faithful conditional expectation such that
 $\tau\circ\mathbb{E}=\tau$.
\begin{definition}\label{subdiag}
 A semifinite subdiagonal algebra of $\mathcal M$ with
 respect to $\mathbb{E}$ is a $w^{*}$-closed subalgebra $\mathcal A$ of $\mathcal M$
 satisfying the following conditions
\begin{enumerate}[\rm(i)]
  \item $\mathcal A+J(\mathcal A)$ is $w^{*}$-dense in $\mathcal M$;
  \item $\mathbb{E}$ is multiplicative on $\mathcal A$, i.e.,
 $\mathbb{E}(xy)=\mathbb{E}(x)\mathbb{E}(y)$ for all $x,y\in \mathcal A$;
  \item $\mathcal A\cap J(\mathcal A)=\mathcal D,$
 where $J(\mathcal A)$ is the family of all adjoint elements of the
 element of $\mathcal A$, i.e., $J(\mathcal A)=\{x^{*}: x\in \mathcal A\}.$
\end{enumerate}
\end{definition}
 The algebra $\mathcal D$ is called the diagonal of $\mathcal A$. It's proved by
 Ji \cite{Ji} (see also \cite{Ex} for the finite case) that a semifinite subdiagonal algebra $\mathcal A$ is
 automatically maximal in the sense that if $\mathcal{B}$ is another
 subdiagonal algebra with respect to $\mathbb{E}$ containing $\mathcal A$, then
 $\mathcal{B}=\mathcal A$. This maximality yields the following useful
 characterization of $\mathcal A$:
 \begin{equation}\label{eq:subdiag} \mathcal A=\{x\in\mathcal M: \tau(xy)=0, \forall y\in \mathcal A_{0}\} \end{equation}
where $\mathcal A_{0}=\mathcal A\cap\ker\mathbb{E}$ (see \cite{A}).

Let $0<p<\infty.$ Then we define the noncommutative $L_{p}$-spaces associated with $(\mathcal M,\tau)$ as follows
 $$L_{p}(\mathcal M)=\{x\in S(\mathcal{M},\tau):\tau(|x|^p)<\infty\}$$
 with the quasi-norm $$\|x\|_p=(\tau(|x|^p))^{\frac{1}{p}}.$$

Recall that $L_{\infty}(\mathcal M)=\mathcal M,$ equipped with the operator norm. For more details on these spaces we refer the reader to \cite{PX2003}. Now we will define noncommutative $H_p$-space (see\cite{B1}).

 \begin{definition}\label{hp}
 For $0<p\leq\infty$ we define the noncommutative $H_p$-space :
 $$H_{p}(\mathcal A)=\emph{closure of}\ \mathcal A\cap L_{p}(\mathcal M)\  \emph{in}\ L_{p}(\mathcal M),$$
 $$H_{p}^{0}(\mathcal A)=\emph{closure of}\ \mathcal A_{0}\cap L_{p}(\mathcal M)\  \emph{in}\ L_{p}(\mathcal M).$$
\end{definition}

 These are so called Hardy spaces associated with semifinite subdiagonal algebra $\mathcal A.$ They are noncommutative extensions of the classical Hardy
 space on the torus $\mathbb{T}.$ These
noncommutative Hardy spaces have received a lot of attention since
Arveson's pioneer work. For references see
\cite{A, B1, B2, BS, BTD, BL12008, BL2006, P, MW, PX2003, S, STZ, T1, T2}, whereas more
references on previous works can be found in the survey paper
\cite{PX2003}.

  The theory of vector-valued
 noncommutative $L_{p}$-spaces are introduced by Pisier in
 \cite{P1998} for the case $\mathcal M$ is hyperfinite. Junge introduced these spaces for general setting in \cite{J} (see also \cite{DJ, D, JXu}). Let $0< p,\;r,\;s\leq\infty$ such that $1/p=1/r+1/s$. Define the space
 $$L_{p}^{(r,s)}(\mathcal M,\ell_{\infty})$$
of all sequences $g=\{g_{n}\}_{n\geqslant1}$ of operators in $L_{p}(\mathcal M)$ for which there is a bounded sequence $x=\{x_{n}\}_{n\geqslant1}$ in $\mathcal M$ and operators $y\in L_{r}(\mathcal M)$ and $z\in L_{s}(\mathcal M)$ such that
$$
g_n=yx_nz,\quad \forall\; n\ge1.
$$
Put
$$
\|g\|_{p;(r,s)}=\inf\|y\|_r\sup_n\|x_n\|_\infty\|z\|_s,
$$
where the infimum runs over all possible decompositions of $g=\{g_{n}\}_{n\geqslant1}$ as above. Then $L_{p}^{(r,s)}(\mathcal M,\ell_{\infty})$ is a Banach space whenever $r,s\ge2$ (see \cite{DJ}). Similar way one can prove $L_{p}^{(r,s)}(\mathcal M,\ell_{\infty})$ is a quasi-Banach space whenever $r,s>0$. We put
$$
L_{p}^{right}(\mathcal M,\ell_{\infty})=L_{p}^{(\infty,p)}(\mathcal M,\ell_{\infty}), \quad L_{p}^{left}(\mathcal M,\ell_{\infty})=L_{p}^{(p,\infty)}(\mathcal M,\ell_{\infty})
$$
and
$$
L_{p}(\mathcal M,\ell_{\infty})=L_{p}^{(2p,2p)}(\mathcal M,\ell_{\infty}).
$$

Given  $0<p\leq \infty$,  a sequences $x=\{x_{n}\}_{n\geqslant1}$ belongs to  $L_{p}(\mathcal M;\ell_{1})$ if  there are $\{u_{kn}\}_{k,n\geqslant1}$ and
$\{v_{nk}\}_{n,k\geqslant1}$ in $L_{2p}(\mathcal M)$  such that
$$ x_{n}=\sum_{k=1}^{\infty}u_{kn}v_{nk}$$
for all $n$ and
$$ \sum_{k,n=1}^{\infty}u_{kn}u_{kn}^{*}\in L_{p}(\mathcal M) \quad
 \sum_{n,k=1}^{\infty}v_{nk}^{*}v_{nk}
\in L_{p}(\mathcal M).
$$
Here all series are required to be convergent in $L_p(\mathcal M)$ (relative to the w*-topology
in the case of $p =\infty$). $L_p(\mathcal M; \ell_1)$ is a quasi-Banach space (Banach space whenever $p\ge1$)  when equipped with the norm

\begin{eqnarray*} \|x\|_{L_{p}(\mathcal M;\ell_{1})}=\inf
\{\|\sum_{k,n=1}^{\infty}u_{kn}u_{kn}^{*}\|_{p}^{1/2}\cdot\|\sum_{n,k=1}^{\infty}v_{nk}^{*}v_{nk}\|_{p}^{1/2}\},
\end{eqnarray*} where the infimum runs over all decompositions of $(x_n)$ as
above.

 We now  define the spaces $H^{(r,s)}_{p}(\mathcal A;\ell_{\infty})$ and
$H_{p}(\mathcal A;\ell_{1})$ (see \cite{BTD} for the finite case) by a similar way.

\begin{definition}\label{Hp-space} Let $0< p,\;
r,\;s\leq\infty$ such that $1/p=1/r+1/s.$
\begin{enumerate}[\rm(i)]
  \item  We define $H_{p}^{(r,s)}(\mathcal A,\ell_{\infty})$ as the space of all
sequences $g=\{g_{n}\}_{n\geqslant1}$ in $H_{p}(\mathcal A)$ which admit a
factorization of the following form:

there are $y\in H_{r}(\mathcal A),$
$z\in H_{s}(\mathcal A)$ and a bounded sequence $x=\{x_{n}\}_{n\geqslant1}\subset \mathcal A$ such
that
\begin{eqnarray*}
 g_{n}=yx_{n}z, \quad\forall n\geqslant1.
\end{eqnarray*}
Put
\begin{eqnarray*}
\|g\|_{H_{p}^{(r,s)}(\mathcal{A},\ell_{\infty})}=\inf\{\|y\|_{H_r(\mathcal A)}\sup_{n}\|x_{n}\|_{H_{\infty}(\mathcal A)}\|z\|_{H_{s}(\mathcal A)}\},
\end{eqnarray*}
where the infimum runs over all factorizations of $\{g_{n}\}_{n\geqslant1}$ as
above. The spaces \begin{eqnarray*}
H_{p}^{right}(\mathcal A;\ell_{\infty})=H_{p}^{(\infty,p)}(\mathcal A;\ell_{\infty})\end{eqnarray*}
and \begin{eqnarray*}
H_{p}^{left}(\mathcal A;\ell_{\infty})=H_{p}^{(p,\infty)}(\mathcal A;\ell_{\infty})\end{eqnarray*}
of all sequences $\{g_{n}\}_{n\geqslant1}$ which allow
uniform factorizations $g_{n}=x_{n}z$ and $g_{n}=yx_{n}$ with
$y,z\in H_{p}(\mathcal A)$ and a bounded sequence $\{x_{n}\}_{n\geqslant1}\subset \mathcal A$,
respectively. Moreover, in the symmetric case put \begin{eqnarray*}
H_{p}(\mathcal A;\ell_{\infty})=H_{p}^{(2p,2p)}(\mathcal A;\ell_{\infty}).\end{eqnarray*}
 \item Let $0< p\leqslant \infty$. We define $H_{p}(\mathcal A;\ell_{1})$ as the
space of all sequences $x=\{x_{n}\}_{n\geqslant1}$ in $H_{p}(\mathcal A)$ which
can be decomposed as \begin{eqnarray*} x_{n}=\sum_{k=1}^{\infty}u_{kn}v_{nk},
\forall n\geqslant1
\end{eqnarray*} for two families $\{u_{kn}\}_{k,n\geqslant1}$ and
$\{v_{nk}\}_{n,k\geqslant1}$ in $H_{2p}(\mathcal A)$ such that
 \begin{eqnarray*} \sum_{k,n=1}^{\infty}u_{kn}u_{kn}^{*}\in L_{p}(\mathcal M) \, and \,
 \sum_{n,k=1}^{\infty}v_{nk}^{*}v_{nk}
\in L_{p}(\mathcal M).\end{eqnarray*}
Here all series are required to be convergent in $L_p(\mathcal M)$ (relative to the w*-topology
in the case of $p =\infty$).
It is equipped with the norm
\begin{eqnarray*} \|x\|_{H_{p}(\mathcal A;\ell_{1})}=\inf
\{\|\sum_{k,n=1}^{\infty}u_{kn}u_{kn}^{*}\|_{L_p(\mathcal M)}^{1/2}\|\sum_{n,k=1}^{\infty}v_{nk}^{*}v_{nk}\|_{L_p(\mathcal M)}^{1/2}\},
\end{eqnarray*} where the infimum runs over all decompositions of $x$ as
above.
\end{enumerate}

\end{definition}

\section{Contractibility of $\widetilde{\mathbb{E}}$ on $H^{(r,s)}_{p}(\mathcal A;\ell_{\infty})$ and
$H_{p}(\mathcal A;\ell_{1})$ spaces}\label{main results}

In this section, we define a conditional expectation on $H^{(r,s)}_{p}(\mathcal A;\ell_{\infty})$ and
$H_{p}(\mathcal A;\ell_{1})$ spaces defined in the Definition \ref{Hp-space} and prove that it is a contractive projection on such spaces.
First, we need the following result.
\begin{lemma}\label{multiplicative}
$\mathbb{E}$ is multiplicative on Hardy spaces. More precisely, $\mathbb{E}(xy)=
\mathbb{E}(x)\mathbb{E}(y)$ for all $x\in H_p(\mathcal A)$ and $y\in H_q(\mathcal A)$ with $0<p,q\leqslant\infty.$
\end{lemma}
\begin{proof}
Let $x\in H_p(\mathcal{A})$ and $y\in H_q(\mathcal{A}).$ Then, by Definition \ref{hp} there exist sequences $x_n\in L_p(\mathcal{M})\cap\mathcal{A}$ and $y_n\in L_q(\mathcal{M})\cap\mathcal{A}$ such that $\lim_{n\rightarrow\infty}\|x_n-x\|_{L_p(\mathcal{M})}=0$ and $\lim_{n\rightarrow\infty}\|y_n-y\|_{L_q(\mathcal{M})}=0.$ Hence, it follows from H\"older inequality (see \cite{PX2003}) that $x_ny_n\in L_r(\mathcal{M})\cap\mathcal{A},$ where $r$ is determined by $\frac1r=\frac1p+\frac1q.$ Letting $n\rightarrow\infty$, we obtain $xy\in H_r(\mathcal{A}).$ Thus, $\mathbb{E}(xy)$ is well defined. Then, the result follows immediately from the multiplicativity of $\mathbb{E}$ on $\mathcal{A}$ and \cite[Proposition 3.1]{B1}.
\end{proof}

Define a conditional expectation on $H_{p}^{(r,s)}(\mathcal A;\ell_{\infty})$ and $H_{p}(\mathcal A;\ell_{1})$ spaces as follows $$\widetilde{\mathbb{E}}:h=\{h_{n}\}_{n\geqslant1}\rightarrow\Big\{\mathbb{E}\left(h_{n}\right)\Big\}_{n\geqslant1},$$
where $\mathbb{E}$ is the conditional expectation on $H_{p}(\mathcal A)$ such that $\mathbb{E}:H_{p}(\mathcal{A})\rightarrow L_{p}(\mathcal{D}), \, 0<p\leq\infty.$
\begin{theorem}\label{elinfty}
Let $0<p,r,s\leq\infty$. Then
$$
\|\widetilde{\mathbb{E}}(h)\|_{L_{p}^{(r,s)}(\mathcal D;\ell_{\infty})}\le \|h\|_{H_{p}^{(r,s)}(\mathcal A;\ell_{\infty})},\quad\forall\;h=\{h_{n}\}_{n\geqslant1}\in
H_{p}^{(r,s)}(\mathcal A;\ell_{\infty}).
$$
\end{theorem}

\begin{proof} Let $h=\{h_{n}\}_{n\geqslant1}\in
H_{p}^{(r,s)}(\mathcal A;\ell_{\infty})$. Then for  $\varepsilon>0$ there
exist $x\in H_{r}(\mathcal A),\; y\in H_{s}(\mathcal A)$ and a bounded sequence $\{z_{n}\}_{n\geqslant1}\subset\mathcal A$  such that for all $n$,
$h_{n}=yx_{n}z$, and
\begin{eqnarray*}\|\{h_{n}\}_{n\geqslant1}\|_{H_{p}^{(r,s)}(\mathcal A;\ell_{\infty})}+\varepsilon\geqslant\|y\|_{H_{r}(\mathcal A)}\sup_{n}\|x_{n}\|_{H_{\infty}(\mathcal A)}\|z\|_{H_{s}(\mathcal A)}.\end{eqnarray*}
Hence, by Lemma~\ref{multiplicative} and \cite[Proposition 3.1]{B1}, we have

\begin{eqnarray*}
\mathbb{E}(h_{n})=\mathbb{E}(yx_{n}z)=\mathbb{E}(y)\mathbb{E}(x_{n})\mathbb{E}(z),
\end{eqnarray*}
where
\begin{eqnarray*}\mathbb{E}(y)\in L_{r}(\mathcal D),\,\,\ \mathbb{E}(x_{n})\in\mathcal D, \,\,\,\ \mathbb{E}(z)\in
L_{s}(\mathcal D)\end{eqnarray*} and \begin{eqnarray*}\|\mathbb{E}(y)\|_{L_{r}(\mathcal D)}\leq\|y\|_{H_{r}(\mathcal A)},\|\mathbb{E}(x_{n})\|_{L_{\infty}(\mathcal D)}\leq\|x_{n}\|_{H_{\infty}(\mathcal A)},\|\mathbb{E}(z)\|_{L_{s}(\mathcal D)}\leq\|z\|_{H_{s}(\mathcal A)}.\end{eqnarray*}
Therefore,

\begin{eqnarray*}
  \|\widetilde{\mathbb{E}}(h)\|_{L_{p}^{(r,s)}(\mathcal D;\ell_{\infty})} &=& \|\{\mathbb{E}(h_n)\}_{n\geqslant1}\|_{L_{p}^{(r,s)}(\mathcal D;\ell_{\infty})} \\
  &\leq& \|\mathbb{E}(y)\|_{L_{r}(\mathcal D)}\sup_{n}\|\mathbb{E}(x_{n})\|_{L_{\infty}(\mathcal D)}\|\mathbb{E}(z)\|_{L_{s}(\mathcal D)}  \\
   &\leq&\|y\|_{H_{r}(\mathcal A)}\sup_{n}\|x_{n}\|_{H_{\infty}(\mathcal A)}\|z\|_{H_{s}(\mathcal A)}\\
   &\leq&\|\{h_{n}\}_{n\geqslant1}\|_{H_{p}^{(r,s)}(\mathcal A;\ell_{\infty})}+\varepsilon \\
   &=&\|h\|_{H_{p}^{(r,s)}(\mathcal A;\ell_{\infty})}+\varepsilon.
\end{eqnarray*}
Letting $\varepsilon\rightarrow0$ we obtain the desired
inequality.
\end{proof}

Let $x=\{x_n\}_{n\geqslant0}$ be a finite sequence in $L_p(\mathcal{M}),$ define
$$\left\|x\right\|_{L_{p}(\mathcal{M},\ell_2^C)}:=\left\|\left(\sum_{n\geqslant0}|x_{n}|^{2}\right)^{1/2}\right\|_{L_{p}(\mathcal{M})}.$$

\begin{lemma}\label{expect} Let $0<p\le\infty$. Then
$$
\|(\sum_{n=0}^{\infty}|\mathbb{E}(x_{n})|^{2})^{1/2}\|_{L_{p}(\mathcal D)}\leq\|(\sum_{n=0}^{\infty}|x_{n}|^{2})^{1/2}\|_{H_{p}(\mathcal A)},\quad \forall\; \{x_n\}_{n\geqslant1}\subset H_{p}(\mathcal A).
$$
\end{lemma}
\begin{proof} Let $\|(\sum_{n=0}^{\infty}|x_{n}|^{2})^{1/2}\|_{H_{p}(\mathcal A)}<\infty$ and let $\mathcal{N}=M_n(\mathcal{D})$ be the algebra of $n\times n$ matrices with entries from
$\mathcal{D}$ and $\mathcal{B}=M_n(\mathcal{A})$ be the algebra of $n\times n$ matrices with entries from $\mathcal{A}.$
For $x\in\mathcal{N}$ with entries $x_{i,j},$ define $\Phi(x)$ to be the matrix with entries $\mathbb{E}(x_{i,j})$
and $\nu(x)=\sum_{i=1}^{n}\tau(x_{i,i}).$ Then $(\mathcal{N},\nu)$ is a semifinite von Neumann algebra and $\mathcal{B}$ is a semifinite subdiagonal algebra of $(\mathcal{N},\nu)$. Let $x'=\{(x_1,\ x_2,\cdots,\ x_n,\ 0,\cdots)\}.$ Then by \cite[Proposition 3.1]{B1}, we have
\begin{equation}\label{contradiction}
\begin{split}
  \|(\sum_{k=0}^{n}|\mathbb{E}(x_{n})|^{2})^{1/2}\|_{L_{p}(\mathcal D)}&= \|\Phi(T(x'))\|_{L_{p}(\mathcal N)}\leqslant\|T(x')\|_{H_{p}(\mathcal B)}\\
  &=\|(\sum_{k=0}^{n}|x_{n}|^{2})^{1/2}\|_{H_{p}(\mathcal A)},\quad \forall\; \{x_n\}_{n\geqslant1}\subset H_{p}(\mathcal A)
\end{split}
\end{equation}

where $T$ is the map (see also \cite{B2}):
\begin{eqnarray*}
T:\{x_{n}\}_{n\geqslant1}\mapsto\left(
                                    \begin{array}{ccc}
                                      x_{1} & 0 & \cdots \\
                                      x_{2} & 0 & \cdots \\
                                      \vdots & \vdots& \ddots\\
                                    \end{array}
                                  \right).
\end{eqnarray*}

Letting $n\rightarrow\infty$ in \eqref{contradiction}, we obtain desired inequality.
\end{proof}

\begin{theorem}\label{el1}
Let $0<p\leq\infty$. Then
$$
\|\widetilde{\mathbb{E}}(h)\|_{L_{p}(\mathcal D;\ell_{1})}\leqslant\|h\|_{H_{p}(\mathcal A;\ell_{1})},\quad \forall\;h=\{h_{n}\}_{n\geqslant1}\in
H_{p}(\mathcal A;\ell_{1}).
$$
\end{theorem}
\begin{proof} Let $\{h_{n}\}_{n\geqslant1}\in
H_{p}(\mathcal A;\ell_{1}),$ then for  $\varepsilon>0$, there are
$\{u_{kn}\}_{k,n\geqslant1}$ and $\{v_{nk}\}_{n,k\geqslant1}$ in $H_{2p}(\mathcal A)$
such that \begin{eqnarray*}h_{n}=\sum_{k=1}^{\infty}u_{kn}v_{nk},\quad \forall
n\geqslant1, \end{eqnarray*} and
$(\sum_{k,n=1}^{\infty}u_{kn}u_{kn}^{\ast})^{\frac{1}{2}},\;(\sum_{n,k=1}^{\infty}v_{nk}^{\ast}v_{nk})^{\frac{1}{2}}\in
L_{2p}(\mathcal M)$, and
\begin{eqnarray*}\|\left\{h_{n}\right\}_{n\geqslant1}\|_{H_{p}(\mathcal A;\ell_{1})}+\varepsilon\geqslant\|\sum_{k,n=1}^{\infty}u_{kn}u_{kn}^{\ast}\|_{H_{p}(\mathcal A)}^{\frac{1}{2}}\cdot\|\sum_{k,n=1}^{\infty}v_{nk}^{\ast}v_{nk}\|_{H_{p}(\mathcal A)}^{\frac{1}{2}}.\end{eqnarray*}

Hence, by Lemma~\ref{multiplicative}
\begin{eqnarray*}
\mathbb{E}(h_{n})=\mathbb{E}(\sum_{k=1}^{\infty}u_{kn}v_{nk})=\sum_{k=1}^{\infty}\mathbb{E}(u_{kn}v_{nk})=\sum_{k=1}^{\infty}\mathbb{E}(u_{kn})\mathbb{E}(v_{nk}),\,\,\ \forall \,\ n.\end{eqnarray*}
Since $J(\mathcal A)$ is a subdiagonal algebra of $\mathcal A$,
using Lemma~\ref{expect} we obtain
\begin{eqnarray*}\left\|\widetilde{\mathbb{E}}(h)\right\|_{L_{p}(\mathcal D;\ell_{1})}&=&\left\|\left\{\mathbb{E}\left(h_{n}\right)\right\}_{n\geqslant1}\right\|_{L_{p}(\mathcal D;\ell_{1})} \\
&=&\left\|\left\{\sum_{k=1}^{\infty}\mathbb{E}(u_{kn})\mathbb{E}(v_{nk})\right\}_{n\geqslant1}
\right\|_{L_{p}(\mathcal D;\ell_{1})}\\
&\leqslant&\left\|\sum_{k,n=1}^{\infty}|\mathbb{E}(u_{kn}^{*})|^{2}\right\|_{L_{p}(\mathcal D)}^{1/2}\cdot\left\|\sum_{k,n=1}^{\infty}|\mathbb{E}(v_{nk})|^{2}\right\|_{L_{p}(\mathcal D)}^{1/2}\\
&\leqslant&\left\|\sum_{k,n=1}^{\infty}|u_{kn}^{*}|^{2}\right\|_{H_{p}(\mathcal A)}^{1/2}\cdot\left\|\sum_{k,n=1}^{\infty}|v_{nk}|^{2}\right\|_{H_{p}(\mathcal A)}^{1/2}\\
&\leqslant&\left\|\left\{h_{n}\right\}_{n\geqslant1}\right\|_{H_{p}(\mathcal A;\ell_{1})}+\varepsilon\\
&=&\left\|h\right\|_{H_{p}(\mathcal A;\ell_{1})}+\varepsilon.
\end{eqnarray*}

Letting $\varepsilon\rightarrow0$ we complete the proof.
\end{proof}

\end{document}